\documentclass[11pt,reqno]{amsart}

\usepackage{amsmath,amssymb,amsthm,amsfonts,mathrsfs,eucal,stmaryrd,tikz-cd,yhmath,datetime}

\usepackage[utf8]{inputenc}

\numberwithin{equation}{section}

\usepackage[a4paper,margin=1.2in]{geometry}
                		
\usepackage{graphicx}				
										
\usepackage{amssymb}

\usepackage[utf8]{inputenc}

\usepackage[polish, english]{babel}

\usepackage[T1]{fontenc}

\usepackage{mathtools}

\usepackage{tikz-cd}

\usepackage{amsthm}

\usepackage[toc]{appendix}

\usepackage{calligra}

\usepackage{hyperref}

\usepackage{array}

\usepackage{makecell}

\newtheorem{thm}{Theorem}[section]
\newtheorem{lem}[thm]{Lemma}

\newtheorem{que}[thm]{Question}

\newtheorem{prop}[thm]{Proposition}
\newtheorem{cor}[thm]{Corollary}

\theoremstyle{remark}
\newtheorem{rmk}[thm]{Remark}

\newtheorem{lista}[thm]{List}

\newcommand{\cA}{{\mathcal A}}

\newcommand{\cR}{{\mathcal R}}
\newcommand{\cQ}{{\mathcal Q}}
\newcommand{\cS}{{\mathcal S}}

\newcommand{\CC}{{\mathbb C}}

\newcommand{\PP}{{\mathbb P}}

\newcommand{\ZZ}{{\mathbb Z}}

\newcommand{\Ox}{{\mathscr{O}_{X}}}

\newcommand{\Dx}{{\mathscr{D}_{X}}}

\newcommand{\Dm}{{\mathscr{D}_{X}^{(e)}}}
\newcommand{\Dxm}{{\mathscr{D}_{X}^{\leq m}}}

\newcommand{\eE}{{\mathscr{E}}}
\newcommand{\Oq}{{\mathscr{O}_{Q_{2m}}}}
\newcommand{\Sp}{{\Sigma_{+}}}
\newcommand{\Sm}{{\Sigma_{-}}}

\title{$D$-affinity of quadrics revisited} 

\author{Feliks Rączka}

\address{Institute for Advanced Study, 1 Einstein Drive, 08540, Princeton, NJ
  }

\email{fraczka@impan.pl}

\date{\today}

\begin{document}

\maketitle

\begin{abstract}
Let $K$ be aa algebraically closed field of characteristic $p\geq3$ and let $Q_{n}\subset\PP^{n+1}_{K}$ be a smooth quadric hypersurface. We show that if $n=2m\geq4$ then $Q_{n}$ is not $D$-affine. In particular, we show the grassmannian $\textnormal{Gr}(2,4)$ is not $D$-affine, which gives an example of a non $D$-affine flag variety of minimal possible dimension in characteristic $p\geq3$. Our result complements previous work of A.\ Langer, who showed that if $p\geq n=2m+1$ then $Q_{n}$ is $D$-affine.
\end{abstract}
\section{Introduction}
\subsection{The main result}
Let $X$ be a smooth projective variety over an algebraically closed field $K$ and let $\Dx$ be the sheaf over $X$ of $K$-linear differential operators (in the sense of Grothendieck). Recall that a $D$-\textit{module} is a left $\Dx$-module that is quasi-coherent as a left $\Ox$-module. Further, recall that $X$ is called $D$-\textit{quasi-affine} if every $D$-module is generated (as a $\Dx$-module) by its global sections, and that $X$ is called $D$-\textit{affine} if it is $D$-quasi-affine and the vanishing $H^{i}(X,M)=0$ holds for all $i>0$ and all $D$-modules $M$. We denote by $Q_{n}$ the (unique up to an isomorphism) smooth quadric hypersurface in $\PP^{n+1}$. The objective of this note is to show the following. 
\begin{thm}\label{Thm1}
Assume that $\textnormal{char }K=p\geq 3$ and $m\geq 2$. Then the even-dimensional smooth quadric hypersurface $Q_{2m}$ is not $D$-affine.
\end{thm}
\begin{rmk}
It is follows from the work of Beilinson--Bernstein \cite{Beilinson-Bernstein} and Holland-Polo \cite{Holland-Polo} that over a field of of characteristic zero $Q_{n}$ is $D$-affine for all $n\geq 1$. In the case of positive characteristic, A.\ Langer \cite{Langer_Quadrics} showed that odd-dimensional quadric $Q_{2m+1}$ is $D$-affine provided $p\geq 2m+1$.
\end{rmk}
\noindent
\noindent
The above result is best appreciated in the broader context of the problem of classification of $D$-affine flag varieties. Below, we briefly survey on this problem to motivate our work. Then, we state some consequences of Theorem \ref{Thm1}. To end this section, we explain the key ideas behind the proof of Theorem \ref{Thm1}.

\subsection{$D$-affinity of flag varieties}

Let $G$ be a semi-simple, simply connected algebraic group, and let $P\subset G$ be a parabolic subgroup. In what follows, we refer to the homogeneous variety $G/P$ as a \textit{flag variety}. When $P$ is a Borel subgroup, we say that the corresponding flag variety is \textit{full}. It is well-known that over $K=\CC$ all flag varieties are $D$-affine. For full flag varieties, this has been established in the celebrated paper of Beilinson--Bernstein \cite{Beilinson-Bernstein}, and generalized to arbitrary flag varieties by Holland--Polo in \cite{Holland-Polo}. When $\textnormal{char }K=p>0$, classification of $D$-affine flag varieties remains an open problem. 
\medskip

From now on, we assume that $\textnormal{char }K=p>0$.

\begin{lista}\label{List1}
The following is known about $D$-affinity of flag varieties over fields of positive characteristic $p$.
\begin{enumerate}
    \item The projective space $\PP^{n}$ is $D$-affine (B.\ Haastert \cite[3.1. Satz]{Haastert}),
    \item The smooth quadric hypersurface of odd dimension $Q_{2m-1}\subset \PP^{2m}$ is $D$-affine,
    provided $p\geq 2n-1$ (A.\ Langer \cite[Corollary 0.3]{Langer_Quadrics}),
    \item The full flag variety $\textnormal{SL}_{3}/B$ is $D$-affine (B.\ Haastert \cite[4.5.4. Satz]{Haastert}),
    \item The full flag variety $\textnormal{Sp}_{4}/B$ is $D$-affine (A.\ Samokhin \cite{Samokhin_D-affine}),
    \item The grassmannian $\textnormal{Gr}(2,5)$ is \textit{not} $D$-affine for any $p>0$ (Kashiwara--Lauritzen \cite{Kashiwara-Lauritzen}),
    \item If $P_{1}\subset P_{2}$ are two parabolic subgroups and $G/P_{1}$ is $D$-affine then $G/P_{2}$ is $D$-affine (this follows from Langer's \cite[Theorem 0.2(4)]{Langer_D-affine}),
    \item $X\times Y$ is $D$-affine if and only if both $X$ and $Y$ are $D$-affine.
\end{enumerate}
\end{lista}
\noindent
To the best of our knowledge, prior to this article there was no examples of flag varieties that are known (not) to be $D$-affine and are not constructed by applying operations (6)-(7) to the flag varieties (1)-(5) above. In particular, we believe that Theorem \ref{Thm1} is the first example of a flag variety that is not $D$-affine and does not admit a fibration to $\textnormal{Gr}(2,5)$.

\subsection{Consequences of Theorem \ref{Thm1}}

Below, we list some consequences of Theorem \ref{Thm1} for the problem of classification of $D$-affine flag varieties.
\medskip

First, using Langer's result on fibrations of $D$-affine varieties \cite[Theorem 0.2(4)]{Langer_D-affine} we construct further examples of non-$D$-affine flag varieties.
\begin{cor}\label{Cor1}
Assume that $\textnormal{char }K=p\geq3$ and $m\geq2$. Let $Q_{2m}\subset \PP^{2m+1}$ be a smooth quadric hypersurface, let $Y$ be a smooth projective variety, and let $f:Y\to Q_{2m}$ be a surjective morphism such that $f_{*}\mathscr{O}_{Y}=\mathscr{O}_{Q_{2m}}$. Then $Y$ is not $D$-affine.
\end{cor}
\noindent
In the context of flag varieties, the above corollary is (6) from List \ref{List1}. For example, it shows that if $G$ is a semi-simple, simply connected group of type $D_{m+1}$, $m\geq3$ and $B\subset G$ is a Borel subgroup then the full flag variety $G/B$ is not $D$-affine. We leave a more detailed discussion for Section \ref{Section_More_Counterexamples}. 
\medskip

\noindent
Second, recall that the Pl{\"u}cker embedding $\textnormal{Gr}(2,4)\hookrightarrow\PP^{5}$ realizes $\textnormal{Gr}(2,4)$ as a quadric hypersurface in $\PP^{5}$. In particular, as a special case of Theorem \ref{Thm1}, we obtain the following.
\begin{cor}\label{Cor2}
Assume that $\textnormal{char K}=p\geq 3$. Then the grassmannian $\textnormal{Gr}(2,4)$ is not $D$-affine.
\end{cor}
\begin{rmk}
It is well known that a flag variety of dimension $\leq3$ is isomorphic to one of the following: $\PP^{1},\ \PP^{2},\ \PP^{1}\times\PP^{1},\ \PP^{3},\ \PP^{1}\times\PP^{2},\ \PP^{1}\times\PP^{1}\times\PP^{1},\ \PP(T_{\PP^{2}})=\textnormal{SL}_{3}/B,\ Q_{3}$. It follows from List \ref{List1} that all of these flag varieties are $D$-affine (except possibly $Q_{3}$ when $p=2$), so Corollary \ref{Cor2} provides an example of a non-$D$-affine flag variety of minimal dimension in any characteristic $p\geq3$.
\end{rmk}
\noindent
As a special case of Corollary \ref{Cor1} we obtain the following.
\begin{cor}\label{Cor3}
Assume that $\textnormal{char }K=p\geq3$. Let $P\subset \textnormal{SL}_{4}$ be a parabolic subgroup. If $\textnormal{SL}_{4}/P$ is $D$-affine then either $\textnormal{SL}_{4}/P=\PP^{3}$ or $\textnormal{SL}_{4}/P=\PP(T_{\PP^{3}})$.
\end{cor}

\begin{rmk}
Above, $\PP(T_{\PP^{3}})=\textnormal{Proj}\left(\textnormal{Sym}^{\bullet}T_{\PP^{3}}\right)$ is the Grothendieck projectivisation of the tangent bundle of $\PP^{3}$ that can be also identified with the incidence variety of lines contained in hypersurfaces in $K^{\oplus 4}$, the so-called \textit{Incidence Correspondence}. We remark that while $\PP^{3}$ is known to be $D$-affine thanks to Haastert's result mentioned above, the $D$-affinity of $\PP(T_{\PP^{3}})$ remains an open question.
\end{rmk}

\subsection{Proof strategy and outline of the paper}
Let us end this introductory section by highlighting the key points in the proof of Theorem \ref{Thm1}, and outlining the content of the remainder of the paper.
\medskip

It is well known (\cite[Lemma 1.5]{Langer_D-affine}) that if $X$ is a $D$-quasi-affine variety then $X$ is $D$-affine if and only if $H^{i}(X,\Dx)=0$ for all $i>0$. It has been shown by A.\ Langer \cite[Proposition 1.7]{Langer_Quadrics} that every quadric is $D$-quasi-affine, so to determinate its $D$-affinity we only have to compute cohomology of $\Dx$. Although we do not prove it here, as it is not needed in the proof of Theorem \ref{Thm1}, we remark that it is not hard to show the vanishing
\[
H^{i}(Q_{2m},\mathscr{D}_{Q_{2m}})=0\qquad(i\geq2),
\]
so $Q_{2m}$ is $D$-affine if and only if $H^{1}(Q_{n},\mathscr{D}_{Q_{n}})=0$. Because of the above, we need some nice criterion for detecting non-vanishing of $H^{1}$ of the sheaf of differential operators. We propose the following the following ($\mathsf{F}:X\to X$ is the absolute Frobenius morphism).
\begin{prop}\label{Prop1}
Assume that $\textnormal{char }K=p>0$, and let $X$ be a smooth variety over $K$. The following conditions are equivalent:
\begin{enumerate}
    \item $H^{1}(X,\Dx)=0$.

    \item For every positive integer $t_{0}$, and every short exact sequence of $\Ox$-modules
    \begin{equation}\label{H1_criterion_ses_1}
    0\rightarrow\mathsf{F}^{t_{0}}_{*}\Ox\rightarrow\eE\rightarrow\mathsf{F}^{t_{0}}_{*}\Ox\rightarrow0,
    \end{equation}
    there exists an integer $e\geq0$ such that the short exact sequence 
    \begin{equation}\label{H1_criterion_ses_2}
0\rightarrow\mathsf{F}^{t_{0}+e}_{*}\Ox\rightarrow\mathsf{F}_{*}^{e}\mathscr{E}\rightarrow\mathsf{F}^{t_{0}+e}_{*}\Ox\rightarrow0
    \end{equation}
    obtained by applying $\mathsf{F}^{e}_{*}$ to \textnormal{(\ref{H1_criterion_ses_1})} splits.
\end{enumerate}
\end{prop}
\noindent
We prove Theorem \ref{Thm1} by verifying that (2) in proposition above does not hold . This is possible, in principle, because by the work of A.\ Langer \cite{Langer_Quadrics} and P.\ Achinger \cite{Achinger2} we know the indecomposable direct summands of $\mathsf{F}^{e}_{*}\mathscr{O}_{Q_{n}}$, although we remark that the proof of Theorem \ref{Thm1} requires a substantial amount of extra work. The key idea is that (following the notation of Section \ref{Section_Bundles_on_Q}, where $\cS$ is a sum of the two spinor bundles on $Q_{2m}$) an extension
\[
0\to\cS(-m)^{\oplus \alpha}\to\mathscr{O}_{Q_{2m}}(-m)^{\oplus\rho}\to\cS(-m+1)^{\oplus \beta}\to0
\]
remains non-split (for $p\geq3$) after applying $\mathsf{F}_{*}$ because by the work of Langer and Achinger $\mathsf{F}_{*}\mathscr{O}_{Q_{2m}}(-m)$ is a direct sum of line bundles while both $\mathsf{F}_{*}\cS(-m)$ and  $\mathsf{F}_{*}\cS(-m+1)$ contain indecomposable summands of higher ranks. The proof of Theorem \ref{Thm1} proceeds by showing that because of the above, (2) in Proposition \ref{Prop1} cannot hold.
\medskip

The paper is organized as follows. In Section \ref{Section_Extensions} we recall generalities about extensions that are used throughout the paper. In Section \ref{Section_Differential_Operators} we recall basics facts about differential operators in positive characteristic, and we prove Proposition \ref{Prop1}. In section \ref{Section_Bundles_on_Q} we collect the information about vector bundles on quadrics. The proof of Theorem \ref{Thm1} is given in Section \ref{Section_Proof}. Finally, in Section \ref{Section_More_Counterexamples} we discuss consequences of of Theorem \ref{Thm1} and in Section \ref{Section_Questions} we formulate some open questions motivated by the content of this paper.

\subsection*{Acknowledgments}
This article was written during my stay at the Institute for Advanced Study as Giorgio and Elena Petronio Fellow II. I thank Devlin Mallory for his comments on the preliminary version of this manuscript.

\section{Extensions}\label{Section_Extensions}

In this section, we collect basic facts about extensions. All of the results below are very well known and elementary. The main purpose of this section is to fix the notation and make the exposition of the proof of Theorem \ref{Thm1} more transparent. We assume that $\cA$ be an abelian category with enough injective objects

\subsection{Yoneda's extensions}
Let $M,P\in\cA$. We briefly recall Yoneda's description of $\textnormal{Ext}^{1}_{\cA}(P,M)$ following \cite[Section 3.4]{Weibel}. 
\medskip

\noindent
An \textit{extension} is a short exact sequence
\begin{equation}\label{ExampleOfSES}
0\to M\to N\to P\to0.
\end{equation}
Two extensions are \textit{equivalent} if there exists a commutative diagram
\[
\begin{tikzcd}
0\arrow{r}&M\arrow{r}\arrow{d}{\textnormal{id}}&N\arrow{d}\arrow{r}&P\arrow{r}\arrow{d}{\textnormal{id}}&0\\
0\arrow{r}&M\arrow{r}&N'\arrow{r}&P\arrow{r}&0
\end{tikzcd}
\]
Recall that (\ref{ExampleOfSES}) determines a class in $\textnormal{Ext}_{\cA}^{1}(P,M)$ by taking the image of $\textnormal{Id}_{P}$ under the connecting map
\[
\textnormal{Hom}_{\cA}(P,P)\to \textnormal{Ext}_{\cA}^{1}(P,M)
\]
in the long exact sequence obtained by applying $\textnormal{Hom}_{\cA}(P,-)$ to (\ref{ExampleOfSES}). In what follows, we write  
\begin{equation}\label{ClassOfSES}
\left[
0\to M\to N\to P\to0
\right]
\end{equation}
for the element of $\textnormal{Ext}_{\cA}^{1}(P,M)$ obtained from (\ref{ExampleOfSES}). It is well-known that every element in $\textnormal{Ext}_{\cA}^{1}(P,M)$ is represented by a short exact sequence and that two exact sequences represent the same element if and only if they are equivalent. In particular, (\ref{ClassOfSES}) is zero if and only if (\ref{ExampleOfSES}) splits. 
\subsection{Direct sum of short exact sequences}
Let
\begin{equation}\label{Short1}
0\to M_{1}\xrightarrow{f_{1}} N_{1}\xrightarrow{g_{1}} P_{1}\to 0
\end{equation}
and
\begin{equation}\label{Short2}
0\to M_{2}\xrightarrow{f_{2}} N_{2}\xrightarrow{g_{2}} P_{1}\to 0
\end{equation}
be two short exact sequences. We define its \textit{direct sum} to be the short exact sequence
\begin{equation}\label{SumOfSES}
0\to M_{1}\oplus M_{2}\xrightarrow{f_{1}\oplus f_{2}}N_{1}\oplus N_{2}\xrightarrow{g_{1}\oplus g_{2}}P_{1}\oplus P_{2}\to 0.
\end{equation}
One easily verifies that this construction descends to equivalence classes of short exact sequences and defines an injective homomorphism
\[
\textnormal{Ext}_{\cA}^{1}(P_{1},M_{1})\oplus\textnormal{Ext}_{\cA}^{1}(P_{2},M_{2})\to\textnormal{Ext}_{\cA}^{1}(P_{1}\oplus P_{2},M_{1}\oplus M_{2})
\]
More generally, we can define a direct sum of arbitrary finite collection of short exact sequences in an analogous manner.

\subsection{Morphisms of short exact sequences}
A morphism of two short exact sequences is a morphism in category of complexes, i.e., a commutative diagram
\[
\begin{tikzcd}
0\arrow{r}&M\arrow{r}\arrow{d}&N\arrow{d}\arrow{r}&P\arrow{r}\arrow{d}&0\\
0\arrow{r}&U\arrow{r}&V\arrow{r}&W\arrow{r}&0
\end{tikzcd}
\]
If a short exact sequence
\begin{equation}\label{Short3}
0\to U\to V\to W\to0
\end{equation}
is isomorphic to (\ref{SumOfSES}) then we call (\ref{Short1}) and (\ref{Short2}) \textit{direct summands} of (\ref{Short3}). For future reference, we note the following easy lemma. Its proof is left as an exercise.
\begin{lem}
Consider a short exact sequence 
\begin{equation}\label{Short4}
0\to M\to N\to P\to 0,
\end{equation}
and let 
\[
\eta=[0\to M\to N\to P\to 0]\in\textnormal{Ext}^{1}_{\cA}(P,M).
\]
\begin{enumerate}
    \item $\eta=0$ if and only if \textnormal{(\ref{Short4})} is isomorphic to a split exact sequence.
    \item If \textnormal{(\ref{Short4})} contains a non-split direct summand then $\eta\neq0$.
\end{enumerate}
\end{lem}

\section{Differential operators in positive characteristic}\label{Section_Differential_Operators}
In this section, we recall the basics about sheaves differential operators in positive characteristic. Then, we prove Proposition \ref{Prop1}.
\subsection{The sheaf $\Dx$}
Let $X$ be a smooth projective variety over an algebraically closed field $K$. Recall that the sheaf $\Dx$ of $K$-linear differential operators is defined inductively by letting $\mathscr{D}_{X}^{\leq -1}=0$,
\[
\Dxm=
\left\{
\varphi\in\mathscr{E}nd_{K}(\Ox):\varphi a-a\varphi\in\mathscr{D}_{X}^{\leq m-1}\textnormal{ for all }a\in\Ox
\right\},
\]
and
\[
\Dx=\bigcup_{m\geq0}\Dxm.
\]
Here, we identify $\Ox$ with a subsheaf of $\mathscr{E}nd_{K}(\Ox)$ by letting $a\in\Ox$ act on $\Ox$ by left multiplication. By construction, $\Dx$ is filtered by $\Dxm$. We call this filtration the \textit{order filtration}. 
\subsection{Haastert's $p$-filtration}
If $\textnormal{char }K=p>0$, then $\Dx$ has another filtration, the $p$-\textit{filtration}, introduced in the context of $D$-affinity by B.\ Haastert (cf. \cite[Section 1.2]{Haastert}). It is given by
\[
\Dx=\bigcup_{e\geq0}\Dm,\qquad\Dm=\mathscr{H}om_{\mathscr{O}_{X}^{p^{e}}}
\left(
\Ox,\Ox
\right).
\]
Let $\mathsf{F}:X\to X$ be the absolute Frobenius morphism and denote by $\mathsf{F}^{e}$ the composition of $e$ absolute Frobenii. For the future reference, we remark that as left $\Ox$-modules
\begin{equation}\label{p-filtrationDescription}
\mathsf{F}_{*}^{e}\Dm=\mathscr{H}om_{\Ox}(\mathsf{F}_{*}^{e}\Ox,\mathsf{F}_{*}^{e}\Ox).
\end{equation}
Since $X$ is smooth, $\mathsf{F}$ is finite and flat by Kuntz's theorem. It follows that $\Ox$ is a locally free $\mathscr{O}_{X}^{p^{e}}$-module (or, equivalently, that $\mathsf{F}_{*}^{e}\Ox$ is a locally free $\Ox$-module). Taking cohomology of (\ref{p-filtrationDescription}) we obtain a natural identification
\begin{equation}\label{Ext_Filtration}
H^{i}(X,\Dm)=\textnormal{Ext}^{i}_{\Ox^{p^{e}}}(\Ox,\Ox)=\textnormal{Ext}^{i}_{\Ox}(\mathsf{F}_{*}^{e}\Ox,\mathsf{F}_{*}^{e}\Ox).
\end{equation}
\subsection{Non-vanishing of $H^{1}(X,\Dx)$}
Now, we prove Proposition \ref{Prop1}.
\begin{proof}
As cohomology commutes with filtered colimits, by looking at the $p$-filtration we obtain the identity
\[
H^{1}(X,\Dx)=\varinjlim_{e\to\infty} H^{1}(X,\Dm).
\]
As in \ref{Ext_Filtration} above, we have
\[
H^{1}(X,\Dm)=H^{1}(X,\mathscr{H}om_{\mathscr{O}_{X}^{p^{e}}}(\Ox,\Ox))=\textnormal{Ext}^{1}_{\Ox^{p^{e}}}(\Ox,\Ox).
\]
As in Section \ref{Section_Extensions}, the right hand side can be identified with the Yoneda group of classes of $\mathscr{O}_{X}^{p^{e}}$-linear short exact sequences
\begin{equation}\label{pe-linearSES}
\left[
0\to\Ox\to\widetilde{\eE}\to\Ox\to0
\right],
\end{equation}
and the map $\textnormal{Ext}^{1}_{\Ox^{p^{e}}}(\Ox,\Ox)\to\textnormal{Ext}^{1}_{\Ox^{p^{e+1}}}(\Ox,\Ox)$ induced by the inclusion $\Dm\subset\mathscr{D}_{X}^{e+1}$ maps (\ref{pe-linearSES}) to the class of the same short exact sequence, but now considered in the category of $\mathscr{O}_{X}^{p^{e+1}}$-modules. It follows from the definition of $\varinjlim$ and the above description that giving a non-zero element of $H^{1}(X,\Dx)$ is equivalent to giving for some $t_{0}\geq0$ a short exact sequence of $\Ox^{p^{t_{0}}}$-modules
\[
0\to\Ox\to\widetilde{\eE}\to\Ox\to0
\]
that is non-split in the category of $\mathscr{O}_{X}^{p^{t_{0}+e}}$-modules for all $e\geq0$. Finally, every $\Ox^{p^{e}}$-module $\widetilde{\eE}$ can be treated as an $\Ox$-module $\eE$ if we let $\eE=\widetilde{\eE}$ as sheaves of abelian groups and declare
\[
f.m=f^{p^{e}}m\qquad(f\in\Ox,m\in\eE).
\]
Under this identification $\Ox$ corresponds to $\mathsf{F}_{*}^{e}\Ox$. The equivalence of (1) and (2) in Proposition \ref{Prop1} follows from that.
\end{proof}

\section{Vector Bundles on Quadrics}\label{Section_Bundles_on_Q}
In this section, we recall the properties of vector bundles on quadrics that are needed for the proof of Theorem \ref{Thm1}. We mostly follow the exposition of A.\ Langer \cite{Langer_Quadrics}, a more classical reference is \cite{Ottaviani_Q} by G.\ Ottaviani. In the part regarding Frobenius pushforwards we also often refer to \cite{Achinger2} by P.\ Achinger.
\subsection{Quadrics}
We start by fixing the notation and recalling the basic geometric properties of the quadric hypersurface itself.
\medskip

A \textit{smooth quadric hypersurface} is a smooth hypersurface $Q_{n}\subset\PP^{n+1}$ cut-out by a global section of $\mathscr{O}_{\PP^{n+1}}(2)$. It is well known that all such hypersurfaces are isomorphic. In what follows, we assume that $n=2m$, and $m\geq2$. In this case, we may fix projective coordinates
\[
\PP^{2m+1}=\textnormal{Proj}
\left(
K[x_{1},\dots,x_{2m+2}]
\right)
\]
and assume that 
\[
Q_{2m}=\textnormal{Proj}\left(
K[x_{1},\dots,x_{2m+2}]/(q_{m})
\right),
\]
where 
\[
q_{m}(x_{1},\dots,x_{2m+2})=x_{1}x_{2}+x_{3}x_{4}+\dots+x_{2m+1}x_{2m+2}.
\]
We denote by $i:Q_{2m}\hookrightarrow\PP^{2m+1}$ the induced closed embedding. Consider the involution $\alpha$ of the polynomial ring $K[x_{1},\dots,x_{2m+2}]$ given by
\begin{equation}\label{InvolutionAlpha}
\alpha(x_{2i+1})=x_{2i+2},\ \alpha(x_{2i+2})=x_{2i+1}\qquad(0\leq i\leq m).
\end{equation}
Then $\alpha$ induces an involution of $\PP^{2m+1}$ which further restricts to an involution of $Q_{2m}$. By abuse of notation, we also denote the latter involution by $\alpha$.
\subsection{Line bundles on Quadrics}
In this subsection, we quickly recall what is known about line bundles on $Q_{2m}$. Let $\mathscr{O}_{Q_{2m}}(1)=i^{*}\mathscr{O}_{\PP^{2m+1}}(1)$. This is an ample generator of $\textnormal{Pic}(Q_{2m})$. From the divisorial exact sequence
\[
0\to\mathscr{O}_{\PP^{2m+1}}(-2)\xrightarrow{\times q_{m}}\mathscr{O}_{\PP^{2m+1}}\to i_{*}\Oq\to0
\]
one easily concludes that
\begin{align}\label{CohomologyOfLineBundles}
H^{i}(Q_{2m},\mathscr{O}_{Q_{n}}(t))&=0&(1\leq i\leq 2m-1),\\
H^{0}(Q_{2m},\mathscr{O}_{Q_{n}}(t))&=0&(t<0).
\end{align}
In what follows, given a coherent sheaf $\mathscr{M}$ and an integer $t$ we use the standard notation
\[
\mathscr{M}(t)=\mathscr{M}\otimes_{\mathscr{O}_{Q_{2m}}}\mathscr{O}_{Q_{2m}}(1)^{\otimes t}.
\]
\subsection{Spinor bundles and matrix factorization}
Recall that on $Q_{2m}$ there are two naturally defined spinor bundles $\Sp,$ $\Sm$. These may be defined either geometrically (see \cite{Ottaviani_Q}), or by means of representation theory, (see \cite[Section 1.1]{Langer_Quadrics}). Below, we recall Langer's construction of spinor bundles via matrix factorization.
\medskip

Recall that a matrix factorization of a polynomial $f$ with $f(0)=0$ is a pair of matrices $\varphi,\psi$ of the same size $r\times r$ and with polynomial entries, such that $\varphi\psi=\psi\varphi=f.\textnormal{I}_{r}$, where here and later we write $\textnormal{I}_{r}$ for the identity matrix of size $r\times r$. In \cite[Section 2.2]{Langer_Quadrics} A.\ Langer constructed inductively matrices $\varphi_{m},\psi_{m}\in M_{2^{m}}\left(K[x_{1},\dots,x_{2m+1}]\right)$ that give a matrix factorization of $q_{m}$, by letting $\varphi_{0}=(x_{1}),\ \psi_{0}=(x_{2})$ and
\[
\varphi_{m+1}=
\begin{pmatrix}
\varphi_{m} & x_{2m+1}\textnormal{I}_{2^{m}}\\
x_{2m+2}\textnormal{I}_{2^{m}} & -\psi_{m}
\end{pmatrix}
,\quad
\psi_{m+1}=
\begin{pmatrix}
\psi_{m} & x_{2m+1}\textnormal{I}_{2^{m}}\\
x_{2m+2}\textnormal{I}_{2^{m}} & -\varphi_{m}
\end{pmatrix}
.
\]
This allows to define Spinor bundles on $Q_{2m}$ as
\begin{equation}\label{Definition_Sigma-}
\Sm=\textnormal{Coker}
\left(
\mathscr{O}_{\PP^{2m+1}}(-2)^{\oplus 2^{m}}\xrightarrow{\varphi_{m}}\mathscr{O}_{\PP^{2m+1}}(-1)^{\oplus 2^{m}}
\right),
\end{equation}
and
\begin{equation}\label{Definition_Sigma+}
\Sp=\textnormal{Coker}
\left(
\mathscr{O}_{\PP^{2m+1}}(-2)^{\oplus 2^{m}}\xrightarrow{\psi_{m}}\mathscr{O}_{\PP^{2m+1}}(-1)^{\oplus 2^{m}}
\right).
\end{equation}
\subsection{Symmetry}
Recall from (\ref{InvolutionAlpha}) the involution $\alpha$ of $Q_{2m}$. In this subsection, we prove that $\alpha_{*}$ exchanges $\Sp$ with $\Sm$. More precisely we prove the lemma below.
\begin{lem}\label{SymmetryLemma}
In the above notation, the following is true for arbitrary integer $t$.
\begin{enumerate}
    \item $\alpha_{*}\Oq(t)=\Oq(t)$, 
    \item $\alpha_{*}\Sp(t)=\Sm(t)$,
    \item $\alpha_{*}\Sm(t)=\Sp(t)$.
\end{enumerate}
\end{lem}
\begin{proof}
Since $\alpha$ is an automorphism, we have
\[
\alpha_{*}\alpha^{*}\simeq\alpha^{*}\alpha_{*}\simeq\textnormal{Id}.
\]
In particular, it suffices to prove the version of the lemma where $\alpha_{*}$ is replaced with $\alpha^{*}$. First, note that (1) is true for any automorphism of $Q_{2m}$, because $\alpha^{*}$ must induce an isomorphism on $\textnormal{Pic}(Q_{2m})$ and preserve global generation. This also allows us to assume that $t=0$ in the proof of (2) and (3). Write $\widetilde{\alpha}$ for the automorphism of $M_{2^{m}}\left(K[x_{1},\dots,x_{2m+2}]\right)$ induced by $\alpha$. By (\ref{Definition_Sigma-}) and (\ref{Definition_Sigma+}) we have
\[
\alpha^{*}\Sm=\textnormal{Coker}
\left(
\mathscr{O}_{\PP^{2m+1}}(-2)^{\oplus 2^{m}}\xrightarrow{\widetilde{\alpha}\circ\varphi_{m}\circ\widetilde{\alpha}}\mathscr{O}_{\PP^{2m+1}}(-1)^{\oplus 2^{m}}
\right),
\]
and
\[
\alpha^{*}\Sp=\textnormal{Coker}
\left(
\mathscr{O}_{\PP^{2m+1}}(-2)^{\oplus 2^{m}}\xrightarrow{\widetilde{\alpha}\circ\psi_{m}\circ\widetilde{\alpha}}\mathscr{O}_{\PP^{2m+1}}(-1)^{\oplus 2^{m}}
\right).
\]
A straightforward computation shows that
\[
\widetilde{\alpha}\circ\varphi_{m}\circ\widetilde{\alpha}=\psi_{m},\quad \widetilde{\alpha}\circ\psi_{m}\circ\widetilde{\alpha}=\varphi_{m},
\]
so the lemma follows from (\ref{Definition_Sigma-}) and (\ref{Definition_Sigma+}).
\end{proof}
\subsection{Spinor bundles and cohomology}
In this subsection, we recall well known cohomological properties of spinor bundles $\Sp$ and $\Sm$. We use the convenient notation
\[
\cS=\Sp\oplus\Sm.
\]

First, we note that duals of spinor bundles are twisted spinor bundles.
\begin{lem}[{\cite[Section 2]{Langer_Quadrics}}]\label{SpinorLemma1}
Depending on the parity of $m$, either
\begin{enumerate}
\item $\Sm^{\vee}=\Sm(1)$, and $\Sp^{\vee}=\Sp(1)$ when $m$ is even, or
\item $\Sm^{\vee}=\Sp(1)$, and $\Sp^{\vee}=\Sm(1)$ when $m$ is odd.
\end{enumerate}
In any case,
\begin{enumerate}
    \item[(3)] $\cS^{\vee}=\cS(1)$.
\end{enumerate}
\end{lem}
\noindent
Second, we describe morphisms between twisted spinor bundles.
\begin{lem}\label{SpinorLemma2}
We have
\begin{enumerate}
    \item $\textnormal{Hom}_{\mathscr{O}_{Q_{2m}}}(\Sp,\Sp)=\textnormal{Hom}_{\mathscr{O}_{Q_{2m}}}(\Sm,\Sm)=K.\textnormal{Id}$,

    \item $\textnormal{Hom}_{\mathscr{O}_{Q_{2m}}}(\Sp,\Sm)=\textnormal{Hom}_{\mathscr{O}_{Q_{2m}}}(\Sm,\Sp)=0$,

    \item $\textnormal{Hom}_{\mathscr{O}_{{Q}_{2m}}}(\Sp,\Sp(t))=\textnormal{Hom}_{\mathscr{O}_{{Q}_{2m}}}(\Sm,\Sm(t))=0$ for $t<0$.
\end{enumerate}
\end{lem}
\begin{proof}
All of the above follow from the fact that spinor bundles are slope stable (cf. \cite[Theorem 2.1]{Langer_Quadrics}).
\end{proof}
\noindent
Similarly, we collect information about extensions of spinor bundles.
\begin{lem}\label{SpinorLemma3}
We have
\begin{enumerate}
    \item $\textnormal{Ext}^{1}_{\mathscr{O}_{{Q}_{2m}}}(\cS,\cS(t))=0$ for $t\neq -1$,
    \item $\textnormal{Ext}^{1}_{\mathscr{O}_{{Q}_{2m}}}(\Sigma_{+},\Sigma_{+}(-1))=\textnormal{Ext}^{1}_{\mathscr{O}_{{Q}_{2m}}}(\Sigma_{-},\Sigma_{-}(-1))=0$,
    \item $\textnormal{Ext}^{1}_{\mathscr{O}_{{Q}_{2m}}}(\Sigma_{+},\Sigma_{-}(-1))=\textnormal{Ext}^{1}_{\mathscr{O}_{{Q}_{2m}}}(\Sigma_{-},\Sigma_{+}(-1))=K$.
\end{enumerate}
\end{lem}
\begin{proof}
Follows from Lemma \ref{SpinorLemma1} and \cite[Lemma 2.3]{Langer_Quadrics}.
\end{proof}
\noindent
Finally, we get the description of the cohomology of twists of Spinor bundles.
\begin{lem}\label{SpinorLemma4}
    Let $\Sigma_{\pm}$ be a spinor bundle. Then
    \begin{enumerate}
        \item $H^{i}(\mathscr{O}_{Q_{2m}},\Sigma_{\pm}(k))=0$ for all $1\leq i\leq n-1$ and $k\in\ZZ$.

        \item $H^{0}(\mathscr{O}_{Q_{2m}},\Sigma_{\pm}(k))=0$ for $k\leq 0$.

        \item $\Sigma_{\pm}(t)$ is globally generated for $t\geq1$.
    \end{enumerate}
\end{lem}
\begin{proof}
(1) and (2) follows from Theorem \cite[Theorem 1.1]{Langer_Quadrics} and short exact sequences (2.4)-(2.6) in \textit{loc. cit}. and (3) follows from exact sequences (\ref{Definition_Sigma-}) and (\ref{Definition_Sigma+}).
\end{proof}
\subsection{ACM vector bundles}
Let $i:X\hookrightarrow \PP^{N}$ be a smooth hypersurface and let $\mathscr{O}_{X}(1)=i^{*}\mathscr{O}_{\PP^{N}}$. Recall that a vector bundle $\mathscr{E}$ over $X$ is called an \textit{Arithmetically Cohen--Macaulay} vector bundle (or an ACM bundle for short) if the following vanishing holds.
\[
H^{i}(X,\mathscr{E}(k))=0\qquad\textnormal{for all}\ k\in\ZZ\ \textnormal{and}\ 1\leq i\leq \dim X-1.
\]
The classification of ACM bundles on quadrics in well known.
\begin{thm}[{\cite[Section 1.2, Theorem]{Achinger2}}]\label{ACMBundlesLemma}
    Let $\eE$ be an ACM vector bundle over $Q_{n}$. Then $\eE$ is isomorphic to a direct sum of line bundles and twisted spinor bundles.
\end{thm}
\subsection{Frobenius morphism on quadrics}

Let $\mathsf{F}:Q_{2m}\to Q_{2m}$ be the absolute Frobenius morphism and let $\eE$ be an ACM bundle over $Q_{2m}$. Since the Frobenius morphism is affine, the projection formula yields
\[
H^{i}(Q_{2m},(\mathsf{F}_{*}^{e}\eE)(t))=H^{i}(Q_{2m},\mathsf{F}^{e}_{*}\eE(p^{e}t))=H^{i}(Q_{2m},\eE(p^{e}t))=0\quad(t\in\ZZ,\ 1\leq i\leq 2m-1).
\]
It follows that $\mathsf{F}_{*}^{e}\eE$ is an ACM bundle, so it is a direct sum of line bundles and twisted spinor bundles by Theorem \ref{ACMBundlesLemma}. When $\eE=\mathscr{O}_{Q_{2m}}(t)$ or $\eE=\cS(t)$ the decomposition of $\mathsf{F}_{*}^{e}\eE$ has been described by P.\ Achinger \cite{Achinger2} building on the work of A.\ Langer \cite{Langer_Quadrics}. In this subsection, we briefly present these results, focusing on the aspects that are relevant for the proof of Theorem \ref{Thm1}.
\medskip

First, we remark that $\mathsf{F}$ commutes with arbitrary automorphism. In particular,
\begin{equation}\label{aF=Fa}
\alpha\mathsf{F}=\mathsf{F}\alpha.
\end{equation}
As a consequence, we obtain the following.
\begin{lem}\label{FrobeniusLemma1}
There exist integers $a_{j,t,e},b_{j,t,e},c_{j,t,e},d_{j,t,e}$ such that
\[
\mathsf{F}_{*}^{e}\Oq(j)=\bigoplus_{t}\Oq(t)^{\oplus a_{j,t,e}}\oplus\bigoplus_{t}\cS(t)^{\oplus b_{j,t,e}},
\]
and
\[
\mathsf{F}_{*}^{e}\cS(j)=\bigoplus_{t}\Oq(t)^{\oplus c_{j,t,e}}\oplus\bigoplus_{t}\cS(t)^{\oplus d_{j,t,e}}.
\]
\end{lem}
\begin{proof}
As explained above, $\mathsf{F}_{*}^{e}\Oq(j)$ and $\mathsf{F}_{*}^{e}\cS(j)$ are ACM bundles, so they decompose as direct sums of line bundles and twisted spinor bundles by Theorem \ref{ACMBundlesLemma}. However, by lemma \ref{SymmetryLemma} and (\ref{aF=Fa}) we have
\[
\mathsf{F}^{e}_{*}\Oq(j)=\mathsf{F}^{e}_{*}\alpha_{*}\Oq(j)=\alpha_{*}\mathsf{F}^{e}_{*}\Oq(j),
\]
and
\[
\mathsf{F}^{e}_{*}\cS(j)=\mathsf{F}^{e}_{*}\alpha_{*}\cS(j)=\alpha_{*}\mathsf{F}^{e}_{*}\cS(j),
\]
while $\alpha_{*}\Sigma_{\pm}=\Sigma_{\mp}$. Therefore, multiplicities of $\Sp(t)$ and $\Sm(t)$ in both decompositions are equal, and the lemma follows.
\end{proof}
\noindent
In \cite{Achinger2}, P.\ Achinger found a way to determine which of the multiplicities $a_{j,t,e},b_{j,t,e},c_{j,t,e},d_{j,t,e}$ are non-zero. We recall his theorem below.
\begin{thm}[{P.\ Achinger, \cite[Theorem 2]{Achinger2}}]\label{FrobeniusTheorem1}
Assume that $p\geq 3$, $m\geq2$, and $e\geq1$.
\begin{enumerate}
    \item $\mathsf{F}_{*}^{e}\Oq(j)$ contains $\Oq(-t)$ as a direct summand if and only if
    \[
    0\leq j+p^{e}t\leq 2m(p^{e}-1).
    \]
    \item $\mathsf{F}_{*}^{e}\Oq(j)$ contains $\cS(-t)$ as a direct summand if and only if
    \[
    (m-1)(p-1)p^{e-1}\leq j+tp^{e}\leq mp^{e}+(m-1)p^{e-1}-2m.
    \]
    \item $\mathsf{F}_{*}^{e}\cS(j)$ contains $\Oq(-t)$ as a direct summand if and only if
    \[
    1\leq j+p^{e}t\leq 2m(p^{e}-1).
    \]
    \item $\mathsf{F}_{*}^{e}\cS(j)$ contains $\cS(-t)$ as a direct summand if and only if
    \[
    (m-1)(p-1)p^{e-1}+1-\delta_{e,1}\leq j+tp^{e}
    \leq mp^{e}+(m-1)p^{e-1}-2m+\delta_{e,1}.
    \]
\end{enumerate}
\end{thm}
\noindent
For the rest of this subsection, we prove refinements of the direct sum decompositions from Lemma \ref{FrobeniusLemma1} based on Theorem \ref{FrobeniusTheorem1} above.
\begin{lem}\label{FrobeniusLemma2}
Keep the notation from Lemma \textnormal{\ref{FrobeniusLemma1}} and assume moreover that $p\geq3$, and $p^{e-1}\geq\frac{2m}{m-1}$. Then $b_{0,-m,e}\neq 0$ and $b_{0,-m+1,e}\neq0$.
\end{lem}
\begin{proof}
It follows from Theorem \ref{FrobeniusTheorem1} that for $p\geq3$ the bundle $\mathsf{F}^{e}_{*}\Oq$ contains $\cS(-t)$ as a direct summand if and only if
\[
(m-1)p^{e}-(m-1)p^{e-1} \leq  tp^{e}\leq mp^{e}+(m-1)p^{e-1}-2m.
\]
It is straightforward that if $p^{e-1}\geq\frac{2m}{m-1}$ then both $t=m$ and $t=m-1$ satisfy the above inequalities.
\end{proof}
\begin{lem}[P.\ Achinger, A.\ Langer]\label{FrobeniusLemma3}
Assume that $p\geq3$.
\begin{enumerate}
    \item $\mathsf{F}_{*}\cS(-m+1)=\bigoplus_{t}\Oq(t)^{\oplus c_{-m+1,t,1}}\oplus\cS(-m+1)^{\oplus \gamma}$ with $\gamma\neq0$,
    \item $\mathsf{F}_{*}\cS(-m)=\bigoplus_{t}\Oq(t)^{\oplus c_{-m,t,1}}\oplus\cS(-m)^{\oplus \delta}$ with $\delta\neq0$.
\end{enumerate}
\end{lem}
\begin{proof}
It follows from Theorem \ref{FrobeniusTheorem1} that $\mathsf{F}_{*}\cS(j)$ contains $\cS(t)$ as a direct summand if and only if
\[
(m-1)(p-1)\leq -j(p-1)+p(j-t)\leq m(p-1).
\]
Substituting in the above inequalities $j=-m$ and dividing by $(p-1)$ we obtain
\[
m-1\leq m-\frac{p}{p-1}(m+t)\leq m.
\]
Since $t$ must be an integer, the above inequalities are satisfied if and only if $t=-m$. The reasoning for $j=-m+1$ is analogous.
\end{proof}
\noindent
We need to slightly refine the content of the above lemma. For the future reference let us introduce auxiliary notation. For $a,b\in\left\{+,-\right\}$ we denote
\begin{align}\label{MultiplicityNotation}
u_{a}^{b}\stackrel{\textnormal{def.}}{=}&\textnormal{ multiplicity of }\Sigma_{b}(-m)\textnormal{ as a direct summand of }\mathsf{F}_{*}\Sigma_{a}(-m),\\
v_{a}^{b}\stackrel{\textnormal{def.}}{=}&\textnormal{ multiplicity of }\Sigma_{b}(-m+1)\textnormal{ as a direct summand of }\mathsf{F}_{*}\Sigma_{a}(-m+1).
\end{align}

\begin{lem}\label{FrobeniusLemma4}
Assume $p\geq3$. Then the matrices
\[
\begin{pmatrix}
u_{+}^{+}&u_{-}^{+}\\
u_{+}^{-}& u_{-}^{-}
\end{pmatrix}
\qquad
\begin{pmatrix}
v_{+}^{+}&v_{-}^{+}\\
v_{+}^{-}&v_{-}^{-}
\end{pmatrix}
\]
are symmetric and non-zero.
\end{lem}
\begin{proof}
It follows from Lemma \ref{FrobeniusLemma3} that these matrices are non-zero. The symmetry follows from Lemma \ref{SymmetryLemma} and (\ref{aF=Fa}) like in the proof of Lemma \ref{FrobeniusLemma1}.
\end{proof}
\noindent
Finally, recall Langer's result that will be useful later on.
\begin{lem}[{A.\ Langer, \cite[Corollary 4.3]{Langer_Quadrics}}]\label{FrobeniusLemma5}
Assume that $p\geq3$. Then $\mathsf{F}_{*}\Oq(-m)$ is a direct sum of line bundles.
\end{lem}
\begin{proof}
It follows from Theorem \ref{FrobeniusTheorem1} that $\cS(-t)$ is a direct summand of $\mathsf{F}_{*}\Oq(-m)$ if and only if
\[
(m-1)p+1\leq tp\leq mp-1.
\]
The above equation forces that
\[
(m-1)< t< m,
\]
and therefore is never satisfied by an integer $t$.
\end{proof}
\section{Even dimensional quadrics are not $D$-affine}\label{Section_Proof}

In this section, we prove the following result that implies Theorem \ref{Thm1}.
\begin{thm}\label{Thm2}
Assume that $\textnormal{char }K=p\geq 3$ and $m\geq 2$. Then $H^{1}(Q_{2m},\mathscr{D}_{Q_{2m}})\neq 0$.
\end{thm}
\noindent
The proof will be deduced from a series of lemmas below.
\begin{lem}\label{KeyLemma1}
Let 
\begin{equation}\label{Non-split-Sigma}
    0\to \cS^{\oplus \alpha}\to \eE\to\cS(1)^{\oplus\beta}\to0
\end{equation}
be a short exact sequence. Then
\begin{enumerate}
    \item For some $\alpha_{+},\alpha_{-},\beta_{+},\beta_{-},\rho\geq 0$
    \[
    \eE\simeq \Sigma_{+}^{\oplus\alpha_{+}}\oplus\Sigma_{-}^{\oplus\alpha_{-}}\oplus 
    \Sigma_{+}(1)^{\oplus\beta_{+}}\oplus\Sigma_{-}(1)^{\oplus\beta_{-}}\oplus\mathscr{O}_{Q_{2m}}^{\oplus \rho}. 
    \]
    \item \textnormal{(\ref{Non-split-Sigma})} splits if and only if $\rho=0$.
\end{enumerate}
\end{lem}
\begin{proof}
First, we prove (1). We do it in the case when $m$ is even. The proof for odd $m$ is analogous with the only difference coming from the fact that one needs to use Lemma \ref{SpinorLemma1}(2) instead of Lemma \ref{SpinorLemma1}(1). Note that $\eE$ is an extension of ACM bundles, so it is an ACM bundle. By Theorem \ref{ACMBundlesLemma} $\eE$ is a direct sum of line bundles and twisted spinor bundles. Moreover, it follows from Lemma \ref{SpinorLemma4} that
\begin{equation}\label{VanishingForE(t)}
H^{0}(Q_{2m},\eE(t))=0 \qquad \textnormal{for }t<0,
\end{equation}
and therefore
\[
    \eE\simeq\Sigma_{+}^{\oplus\alpha_{+}}\oplus\Sigma_{-}^{\oplus\alpha_{-}}\oplus 
    \Sigma_{+}(1)^{\oplus\beta_{+}}\oplus\Sigma_{-}(1)^{\oplus\beta_{-}}\oplus\mathscr{O}_{Q_{2m}}^{\oplus \rho}\oplus\eE_{1},
\]
where
\[
\eE_{1}=
\bigoplus_{j<0}\left(
\mathscr{O}_{Q_{2m}}(j)^{\oplus \rho_{j}}\oplus\Sigma_{+}(-j)^{\oplus r_{j}}\oplus\Sigma_{-}(j)^{\oplus s_{j}}
\right).
\]
Now, it follows from Lemma \ref{SpinorLemma1} that $\eE^{\vee}$ fits into a short exact sequence
\[
    0\to \cS^{\oplus\beta}\to\eE^{\vee}\to  \cS(1)^{\oplus\alpha}\to0
\]
obtained by dualizing (\ref{Non-split-Sigma}). In particular, by the reasoning above, 
\[
\eE_{1}^{\vee}(-1)=\bigoplus_{j>0}\left(
\mathscr{O}_{Q_{2m}}(j-1)^{\oplus \rho_{j}}\oplus\Sigma_{+}(j)^{\oplus r_{j}}\oplus\Sigma_{-}(j)^{\oplus s_{j}}
\right)
\]
has no global sections. It follows that $\eE_{1}=0$ and we are done.
\medskip

\noindent
Now, we show (2). Clearly if (\ref{Non-split-Sigma}) splits then $\rho=0$, so we assume that $\rho=0$ and we work to show that (\ref{Non-split-Sigma}) splits. Consider the composition
\[
f:\Sp(1)^{\oplus\beta_{+}}\oplus\Sm(1)^{\oplus\beta_{-}}\to\eE\to\cS(1)^{\oplus\beta}
\]
and observe that since
\[
H^{0}(Q_{2m},\Sigma_{\pm})=H^{1}(Q_{2m},\Sigma_{\pm})=0
\]
by Lemma \ref{SpinorLemma4}, $f$ induces an isomorphism on the global sections. It follows that 
$f$ is surjective, because $\cS(1)$ is generated by global sections by Lemma \ref{SpinorLemma4}(3). On the other hand, by comparing dimensions of global sections we see that
\[
\textnormal{rk}
\left(
\Sp(1)^{\oplus\beta_{+}}\oplus\Sm(1)^{\oplus\beta_{-}}
\right)
=
\textnormal{rk}
\left(
\cS(1)^{\oplus\beta}
\right).
\]
Since a surjective morphism of vector bundles of equal rank is an isomorphism, we conclude that $f$ is an isomorphism. This clearly implies that $\eE\to\cS(1)^{\oplus\beta}$ splits, and we are done.
\end{proof}
\begin{lem}\label{KeyLemma2}
Assume that $p\geq3$ and consider an exact sequence
\begin{equation}\label{Non-split-Sigma2}
0\to \cS(-m)^{\oplus \alpha}\to\eE\to \cS(-m+1)^{\oplus \beta}\to0.
\end{equation}
If \textnormal{(\ref{Non-split-Sigma2})} does not split then 
\begin{equation}\label{F-Non-split-Sigma2}
0\to \mathsf{F}_{*}\left(
\cS(-m)
\right)^{\oplus \alpha}
\to
\mathsf{F}_{*}\eE\to
\mathsf{F}_{*}
\left(
\cS(-m+1)
\right)^{\oplus \beta}
\to0
\end{equation}
is isomorphic to a direct sum of some non-split extension
\begin{equation}\label{Non-split-Sigma3}
0\to \cS(-m)^{\oplus \alpha'}\to\eE_{m}'\to \cS(-m+1)^{\oplus \beta'}\to0
\end{equation}
and a trivial extension.
\end{lem}
\begin{proof}
By Lemma \ref{FrobeniusLemma3} we have 

\[
\mathsf{F}_{*}\cS(-m+1)=\bigoplus_{t}\Oq(t)^{\oplus c_{-m+1,t,1}}\oplus\cS(-m+1)^{\oplus \gamma}
\]
with $\gamma\neq0$, and
\[
\mathsf{F}_{*}\cS(-m)=\bigoplus_{t}\Oq(-j)^{\oplus d_{-m,t,1}}\oplus\cS(-m)^{\oplus \delta}
\]
with $\delta\neq0$. It follows from Lemma \ref{SpinorLemma1}(3) Lemma \ref{SpinorLemma3}(1), and Lemma \ref{SpinorLemma4}(1) that for any integer $t$, and for $-m\leq j\leq -m+1$ we have
\[
\textnormal{Ext}^{1}_{\Oq}\left(\Oq(t),\cS(j)\right)=\textnormal{Ext}^{1}_{\Oq}\left(\cS(j),\Oq(t)\right)=\textnormal{Ext}^{1}_{\Oq}\left(\cS(-m),\cS(-m+1)\right)=0.
\]
As a consequence,
\[
\textnormal{Ext}^{1}_{\Oq}(\mathsf{F}_{*}\cS(-m+1)^{\oplus\alpha},\mathsf{F}_{*}\cS(-m)^{\oplus\beta})=\textnormal{Ext}^{1}_{\Oq}(\cS(-m+1),\cS(m))^{\oplus \alpha\beta\gamma\delta},
\]
so (\ref{F-Non-split-Sigma2}) is isomorphic to a direct sum of some extension (\ref{Non-split-Sigma3}) and a trivial extension. It follows that if (\ref{Non-split-Sigma3}) splits then so does (\ref{F-Non-split-Sigma2}). Therefore, to prove the lemma we only need to show that if (\ref{F-Non-split-Sigma2}) splits then so does (\ref{Non-split-Sigma2}).
By Lemma \ref{KeyLemma1}
\[
\eE\simeq \Sigma_{+}(-m)^{\oplus\alpha_{+}}\oplus\Sigma_{-}(-m)^{\oplus\alpha_{-}}\oplus 
    \simeq \Sigma_{+}(-m+1)^{\oplus\beta_{+}}\oplus\Sigma_{-}(-m+1)^{\oplus\beta_{-}}\oplus\mathscr{O}_{Q_{2m}}(-m)^{\oplus \rho},
\]
and (\ref{Non-split-Sigma2}) splits if and only if $\rho=0$, so we assume that (\ref{F-Non-split-Sigma2}) splits and we work to show that $\rho=0$. This is achieved by counting multiplicities of $\Sigma_{\pm}(-m)$ and $\Sigma_{\pm}(-m+1)$ as direct summands of $\mathsf{F}_{*}\eE$. We have to show that
\begin{equation}\label{RankComparison}
\alpha_{+}+\alpha_{-}=2\alpha,\qquad \beta_{+}+\beta_{-}=2\beta,
\end{equation}
because then $\rho=0$ by comparing ranks of $\eE$ and
\[
\Sigma_{+}(-m)^{\oplus\alpha_{+}}\oplus\Sigma_{-}(-m)^{\oplus\alpha_{-}}\oplus\Sigma_{+}(-m+1)^{\oplus\beta_{+}}\oplus\Sigma_{-}(-m+1)^{\oplus\beta_{-}}.
\]
We will only show the first of the equalities (\ref{RankComparison}) as the other one is proven analogously. At this point, we recall the notation (\ref{MultiplicityNotation}) and we observe that by Lemma \ref{FrobeniusLemma3} and Lemma \ref{FrobeniusLemma5}, none of the factors
\[
\mathsf{F}_{*}\Sigma_{\pm}(-m+1),\ \mathsf{F}_{*}\cS(-m+1),\ \mathsf{F}_{*}\Oq(-m)
\]
contains $\Sigma_{\pm}(-m)$ as a direct summand. It follows that if (\ref{F-Non-split-Sigma2}) splits then
\begin{align}
\alpha(u_{+}^{+}+u_{-}^{+})=&\ \alpha_{+}u_{+}^{+}+\alpha_{-}u_{-}^{+}\\
\alpha(u_{+}^{-}+u_{-}^{-})=&\ \alpha_{+}u_{+}^{-}+\alpha_{-}u_{-}^{-},
\end{align}
which can be rewritten as
\[
\begin{pmatrix}
u_{+}^{+}&u_{-}^{+}\\
u_{+}^{-}&u_{-}^{-}
\end{pmatrix}
\begin{pmatrix}
\alpha-\alpha_{+}\\
\alpha-\alpha_{-}
\end{pmatrix}
=0.
\]
By Lemma \ref{FrobeniusLemma4}, the matrix on the left hand side is symmetric and non-zero, so either it is invertible, or all of its entries are equal and non-zero. In any case
\[
2\alpha-\alpha_{+}-\alpha_{-}=0,
\]
and we are done.
\end{proof}
\noindent
Finally, we prove the main result of the paper.
%
%
%
%
%
\begin{proof}[Proof of Theorem \textnormal{\ref{Thm2}}]
By Proposition \ref{Prop1} we only have to show that for some $e>0$ there exists an extension
\begin{equation}\label{Non-split-after-F}
0\to\mathsf{F}_{*}^{e}\Oq\to\eE\to\mathsf{F}_{*}^{e}\Oq\to0
\end{equation}
such that
\begin{equation}\label{Non-split-after-F+t}
0\to\mathsf{F}_{*}^{e+t}\Oq\to\mathsf{F}_{*}^{t}\eE\to\mathsf{F}_{*}^{e+t}\Oq\to0
\end{equation}
does not split for all $t\geq1$. As in Lemma \ref{FrobeniusLemma1} write
\[
\mathsf{F}_{*}^{e}\Oq\simeq
\bigoplus_{t}\Oq(t)^{\oplus a_{_{0,t,e}}}
\oplus
\bigoplus_{t}\cS(t)^{\oplus b_{0,t,e}}.
\]
Then, as in the proof of Lemma \ref{KeyLemma2} we have for some integers $\mu_{k}$
\[
\textnormal{Ext}^{1}_{\Oq}(\mathsf{F}_{*}^{e}\Oq,\mathsf{F}_{*}^{e}\Oq)
=
\bigoplus_{k<0}\textnormal{Ext}^{1}_{\Oq}
\left(
\cS(k+1),\cS(k)
\right)
^{\oplus \mu_{k}}
\]
and $\mu_{m}\neq 0$ for $e\gg0$ by Lemma \ref{FrobeniusLemma2}. It follows that every (\ref{Non-split-after-F}) is isomorphic to a direct sum of extensions
\[
0\to\cS(k)^{\oplus \alpha_{k}}\to\eE_{k}\to \cS(k+1)^{\oplus \beta_{k}}\to0
\]
and some trivial extension. In particular, we can take (\ref{Non-split-after-F}) to be isomorphic to a direct sum of some non-trivial extension
\[
0\to \cS(-m)^{\oplus \alpha}\to\eE_{m}\to \cS(-m+1)^{\oplus \beta}\to0
\]
(which exists by Lemma \ref{SpinorLemma3}(3)) and a split extension. Then for $t=1$ (\ref{Non-split-after-F+t}) is a direct sum of some non-trivial extension
\[
0\to \cS(-m)^{\oplus \alpha'}\to\eE_{m}'\to \cS(-m+1)^{\oplus \beta'}\to0
\]
and a trivial extension by Lemma \ref{KeyLemma2}. This allows to verify the claim for all $t\geq1$ by induction.
\end{proof}

\section{More counter-examples to $D$-affinity of flag varieties}\label{Section_More_Counterexamples}

In this section, we list some consequences of Theorem \ref{Thm1}, extending the discussion in the introduction.
\medskip

First, we recall that a surjective morphisms $f:X\to Y$ of smooth projective varieties is called a \textit{fibration} if $f_{*}\Ox=\mathscr{O}_{Y}$, and we note a result of A.\ Langer.
\begin{thm}[{A.\ Langer, \cite[Theorem 0.2(4)]{Langer_D-affine}}]\label{LangerFibrationTheorem}
Let $f:X\to Y$ be a fibration between smooth projective variates. If $X$ is $D$-affine then so is $Y$.
\end{thm}
\noindent
We note that Corollary \ref{Cor1} is an immediate consequence of the above and Theorem \ref{Thm1}. The next corollary about $D$-affinity of flag varieties also follows trivially from Langer's result.
\begin{cor}\label{LangerFibrationCorollary}
Let $G$ be a semi-simple, simply connected algebraic group, and let $P_{1}\subset P_{2}\subset G$ be two parabolic subgroups. If $G/P_{1}$ is $D$-affine then so is $G/P_{2}$.     
\end{cor}
\noindent
We now explain how to derive Corollaries \ref{Cor2} and \ref{Cor3} from Theorem \ref{Thm1} and Corollary \ref{LangerFibrationCorollary}. As noted in the introduction, we have an isomorphism $Q_{4}\simeq\textnormal{Gr}(2,4)$, because in its Pl{\"u}cker embedding $\textnormal{Gr}(2,4)$ is cut-out by a single quadric equation in $\PP^{5}$. Therefore, $\textnormal{Gr}(2,4)$ is not $D$-affine if $\textnormal{char }K=p\geq3$ by Theorem \ref{Thm1}. Now, fix $E$ to be an $n$-dimensional $K$-vector space and recall that every flag variety for $\textnormal{SL}_{n}$ is of form
\[
\textnormal{Flag}(i_{_{1}}\dots,i_{r};E)=
\left\{
V_{i_{1}}\subset\dots\subset V_{i_{r}}\subset E:\dim V_{i_{r}}=i_{r}
\right\}
\]
where $1\leq r\leq n-1$ and $1\leq i_{1}<\dots<i_{r}\leq n-1$. There are natural surjections
\[
\textnormal{Flag}(i_{_{1}}\dots,i_{j-1},i_{j},i_{j+1},\dots,i_{r};E)\to\textnormal{Flag}(i_{_{1}}\dots,i_{j-1},i_{j+1},\dots,i_{r};E)
\]
given by forgetting the $V_{i_{j}}$ component of the flag, and it is well known that these may also be described as projections $\textnormal{SL}_{n}/P_{1}\to\textnormal{SL}_{n}/P_{2}$, as in Corollary \ref{LangerFibrationCorollary}. If we set $n=4$, then every flag variety for $\textnormal{SL}_{4}$ is isomorphic to one of the following.
\begin{enumerate}
    \item The full flag variety $\textnormal{SL}_{4}/B=\textnormal{Flag}(1,2,3;E)$,

    \item The partial flag variety $\textnormal{Flag}(1,2;E)\simeq \textnormal{Flag}(2,3;E)$,

    \item The incidence correspondence $\textnormal{Flag}(1,3;E)\simeq \PP(T_{\PP^{3}})$,

    \item The grassmannian $\textnormal{Flag}(2;E)\simeq\textnormal{Gr}(2,4)$,

    \item The projective space $\textnormal{Flag}(1;E)\simeq \textnormal{Flag}(3;E)\simeq \PP^{3}$.
\end{enumerate}
As explained above, flag varieties from (1), (2), and (4) are not $D$-affine if $\textnormal{char }K=p\geq3$, because they project to $\textnormal{Gr}(2,4)$ which is not $D$-affine. Therefore, Corollary \ref{Cor3} holds. At this point we remark that prior to our work it was known that (1) all flag varieties for $\textnormal{SL}_{n}$ with $n\leq 3$ are $D$-affine in any characteristic by the work of B.\ Haastert \cite{Haastert}, and (2) In the case $n=5$ the grassmannian $\textnormal{Gr}(2,5)$ is not $D$-affine in positive characteristic by the work of Kashiwara--Lauritzen \cite{Kashiwara-Lauritzen}. Therefore, our result fills a gap between these two cases and supports the expectation that for $n\geq4$ the flag varieties of $\textnormal{SL}_{n}$ will rarely be $D$-affine in positive characteristic (as explained in the introduction, the only known examples of such flag varieties are projective spaces).
\medskip

To end this section, we want to carry a similar analysis for flag varieties in type $D_{m+1}$ with $m\geq 3$, although in this case we cannot be as explicit. First, we briefly recall the general theory. We refer to Jantzen's monograph \cite{Jantzen_book} for more detailed discussion. If $G$ is a semi-simple, simply connected algebraic group then a choice of a Borel subgroup $B\subset G$ and a maximal torus $T\subset B$ determines a set $S$ of simple roots and for any subset $I\subset S$ a parabolic subgroup $B\subset P_{I}$ with the convention that $P_{\emptyset}=B$, and $P_{S}=G$. It is well known that any flag variety of $G$ must be isomorphic to $G/P_{I}$ for some $I$. If $I_{1}\subset I_{2}\subset S$ then $P_{I_{1}}\subset P_{I_{2}}$ and we have a projection $G/P_{I_{1}}\to G/P_{I_{2}}$ as in Corollary \ref{LangerFibrationCorollary}. For a simple root $\alpha\in S$ we let $P(\alpha)$ be the maximal parabolic subgroup corresponding to $I=S\setminus\{\alpha\}$. Now, consider the root system $D_{m+1}$ with $m\geq3$. The corresponding simply connected algebraic group is $G=\textnormal{Spin}(2m+2)$. It is well known that if we label elements of $S=\left\{\alpha_{1},\dots,\alpha_{m+1}\right\}$ as on the Dynkin diagram below then $Q_{2m}\simeq G/P(\alpha_{1})$.
\begin{figure}[h!]
  \begin{tikzpicture}[scale=.4]
    \node[xshift=0.1 cm, yshift=-0.5cm, label]{$\alpha_{1}$} (0 cm,0) circle (.2cm);
    \draw[xshift=0 cm,thick] (0 cm,0) circle (.2cm);
    \draw[thick] (0.3 cm,0) -- +(2.4 cm,0);
    \node[xshift=1.3 cm, yshift=-0.5cm, label]{$\alpha_{2}$} (0 cm,0) circle (.2cm);
    \draw[xshift=2 cm,thick] (1 cm,0) circle (.2cm);
    \draw[dotted, thick] (3.3 cm,0) -- +(2.4 cm,0);
    \draw[xshift=4 cm,thick] (2 cm,0) circle (.2cm);
    \node[xshift=2.7 cm, yshift=-0.5cm, label]{$\alpha_{m-2}$} (0 cm,0) circle (.2cm);
    \draw[xshift=6 cm,thick] (3 cm,0) circle (.2cm);
    \node[xshift=3.9 cm, yshift=-0.5cm, label]{$\alpha_{m-1}$} (0 cm,0) circle (.2cm);
    \draw[thick] (6.3 cm,0) -- +(2.4 cm,0);
    \draw[xshift=9 cm,thick] (30: 0.3) -- (20: 26 mm);
    \draw[xshift=9 cm,thick] (-30: 0.3) -- (-20: 26 mm);
    \draw[xshift=10 cm,thick] (30: 20 mm) circle (.2cm);
    \node[xshift=5.2 cm, yshift=0.5cm, label]{$\alpha_{m}$};
    \draw[xshift=10 cm,thick] (-30: 20 mm) circle (.2cm);
    \node[xshift=5.4 cm, yshift=-0.5cm, label]{$\alpha_{m+1}$};
  \end{tikzpicture}
  \caption{Dynkin diagram of type $D_{m+1}$}
\end{figure}
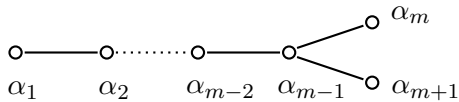
\begin{cor}
With the notation and under the assumptions above let $I\subset S$ be a subset such that $\alpha_{1}\notin I$. If $\textnormal{char }K=p\geq3$ then $\textnormal{Spin}(2m+2)/P_{I}$ is not $D$-affine. In particular, if $B\subset \textnormal{Spin}(2m+2)$ is a Borel subgroup then the full flag variety $\textnormal{Spin}(2m+2)/B$ is not $D$-affine.
\end{cor}
\section{Related problems and some further questions}\label{Section_Questions}
In this final section, we discuss problems related to $D$-affinity of flag varieties and we pose some questions motivated by the results of this paper. First, as explained in the introduction (cf. \cite[Lemma 1.5]{Langer_D-affine}), the process of proving $D$-affinity of a smooth variety $X$ can be split into
\begin{enumerate}
    \item showing that $X$ is $D$-quasi-affine, and

    \item Proving that the vanishing $H^{i}(X,\Dx)=0$ holds for all $i>0$.
\end{enumerate}
Although we do not know whenever all flag varieties are $D$-quasi-affine in positive characteristic, it seems that in the above (2) is much more subtle that (1). For example, Haastert \cite[4.4.2 Korollar]{Haastert} showed that full flag varieties are $D$-quasi-affine, and Langer \cite[Proposition 2.7]{Langer_Quadrics} show that quadrics are $D$-quasi-affine, while the results of this paper show that there exist both quadrics and full flag varieties that are not $D$-affine. The problem of determining the vanishing in (2) is closely related to the problem of determining quasi-exceptionality of $\mathsf{F}_{*}^{e}\Ox$ that has been studied quite intensively, both in the context of flag varieties and some non-homogeneous Fano varieties. Below, we survey the latter problem and formulate some questions.
\medskip

Assume that $X$ is smooth and projective. Recall, that a vector bundle $\eE$ is called \textit{quasi-exceptional} if
\[
\textnormal{Ext}^{i}_{\Ox}(\eE,\eE)=0\qquad(i>0).
\]
Over the field of positive characteristic we have, as in Section \ref{Section_Differential_Operators}
\[
H^{i}(X,\Dm)=\textnormal{Ext}^{i}_{\Ox}(\mathsf{F}_{*}^{e}\Ox,\mathsf{F}_{*}^{e}\Ox),\ H^{i}(X,\Dx)=\varinjlim_{e\to\infty}\textnormal{Ext}^{i}_{\Ox}(\mathsf{F}_{*}^{e}\Ox,\mathsf{F}_{*}^{e}\Ox),
\]
so it follows that if $\mathsf{F}_{*}^{e}\Ox$ is quasi-exceptional for $e\gg0$ then the vanishing (2) holds. If $X$ is $F$-split (for example, a flag variety) then Andersen--Kaneda \cite[Section 2, Propositon]{Andersen-Kaneda} claimed the converse to be true; the vanishing of $H^{i}(X,\Dx)$ for $i>0$ should imply quasi-exceptionality of $\mathsf{F}_{*}^{e}\Ox$ for all $e>0$. Unfortunately, as observed by A.\ Langer \cite[Discussion after Corollary 0.3]{Langer_Quadrics} and D.\ Mallory \cite[Remark 7.4]{Mallory}, the proof of this claim contains a gap and only shows that if $H^{i}(X,\Dm)=0$ then $H^{i}(X,\mathscr{D}_{X}^{(e-1)})=0$.
\medskip

Let us summarize the current state of knowledge.
\begin{lista}\label{List2}
The following is known about quasi-exceptionality of $\mathsf{F}_{*}^{e}\Ox$ on a flag variety $X$.
\begin{enumerate}
    \item In the case of examples (1)-(4) from List \ref{List1}, the authors show that $\mathsf{F}_{*}^{e}\Ox$ is quasi-exceptional for all $e>0$.

    \item The quasi-exceptionality of $\mathsf{F}_{*}^{e}\mathscr{O}_{Q_n}$ on an arbitrary smooth quadric $Q_{n}$ has been determined for all $e$ by A.\ Langer \cite{Langer_Quadrics}, and in the case $p=2$ by P.\ Achinger \cite{Achinger2}. The answer depends on $e,p,n$. Before, A.\ Samokhin \cite{Samokhin_qe}, \cite{Samokhin_quadrics} showed that quasi-exceptionality holds for $e=1$ and all $n$.

    \item Samokhin also showed that $\mathsf{F}_{*}\Ox$ is quasi-exceptional for $X=\PP(T_{\PP^{n}})$.

    \item When $X=\textnormal{Gr}(2,n)$ with $n\geq5$ and $p\geq3$, Readschelders-{\v S}penko--Van den Bergh \cite{Readschelders-Spenko-VdBergh_1} showed that $\mathsf{F}_{*}^{e}\Ox$ is \textit{not} quasi-exceptional for all $e>0$ (except from a finite number of special cases in characteristic $p=3$). 
    \item By the work of M.\ Kaneda \cite{Kaneda_1}, the quasi-exceptionality of $\mathsf{F}^{e}\Ox$ also \textit{fails} for all $e\geq1$ if $p\geq 11$, and $X\simeq G/P$ with $G$ the group of exceptional type $G_{2}$ and $P\subset G$ the parabolic subgroup associated to the short simple root.
\end{enumerate}
\end{lista}
\noindent
We remark that in the context of low dimensional non-homogenous Fano varieties, similar results have been obtained by D.\ Mallory \cite{Mallory} (see also N.\ Hara \cite{Hara}). It is worth pointing out that in the case of examples (4)-(5) in the above list, quasi-exceptionality breaks for $i=1$, so in principle one could try to disprove $D$-affinity of these varieties using our Proposition \ref{Prop1}. In the case of grassmanian, the non-vanishing class in $\textnormal{Ext}^{1}$ is realized by the twist of the tautological extension (cf. \cite[Lemma 16.8]{Readschelders-Spenko-VdBergh_1})
\begin{equation}\label{tautologcaExtension}
0\to \cR(-2)\to V\otimes\Ox(-2)\to \cQ(-2)\to0.
\end{equation}
\begin{que}
Is it true that the extension \textnormal{(\ref{tautologcaExtension})} remains non-split after applying $\mathsf{F}_{*}^{e}$ for arbitrary $e\geq1$?
\end{que}
\noindent
If the answer to the above is affirmative, then it would follow that $\textnormal{Gr}(2,n)$ is not $D$-affine for $n\geq 6$. As we have already mentioned, the case $n=4$ follows from this paper and the case $n=5$ follows from the work of Kashiwara--Lauritzen \cite{Kashiwara-Lauritzen}. 
\medskip

\noindent
As explained above, the first Frobenius pushforward of $\Ox$ is quasi-exceptional for the incidence correspondence $X=\PP(T_{\PP^{n}})$. By the results of the previous section, an answer to the question below would complete the classification of $D$-affine flag varieties for $\textnormal{SL}_{4}$.
\begin{que}
Is $\PP(T_{\PP^{3}})$ $D$-affine?
\end{que}
\noindent
We remark that $\PP(T_{\PP^{2}})\simeq\textnormal{SL}_{3}/B$ is $D$-affine by Haastert's result \cite{Haastert}.
\medskip

\noindent
Finally, we remark that by Langer's result the smooth quadric $Q_{2m+1}$ is $D$-affine if $p\geq 2m+1$, but his approach does not allow to conclude the same in small characteristics. Therefore, it would be interesting to aswer the following question.
\begin{que}
Is $Q_{2m+1}$ $D$-affine if $p\leq 2m$?
\end{que}

\newpage
\bibliographystyle{plain}
\bibliography{Bibliography}

\end{document}